\documentclass[12pt]{amsart}

\usepackage{graphics, amsmath, amsfonts, amssymb, amsthm, amscd, mathrsfs}

\input{xy}
\xyoption{all}

\newtheorem{theorem}{Theorem}
\newtheorem{proposition}[theorem]{Proposition}

\newtheorem{corollary}[theorem]{Corollary}

\theoremstyle{definition}

\newtheorem{remark}[theorem]{Remark}

\newcommand{\C}{\mathbb{C}}
\newcommand{\scrC}{\mathscr{C}}
\newcommand{\scrO}{\mathscr{O}}
\newcommand{\id}{\operatorname{id}}

\title{Absolute neighbourhood retracts \\ and spaces of holomorphic maps \\ from Stein manifolds to Oka manifolds}

\author{Finnur L\'arusson}

\address{School of Mathematical Sciences, University of Adelaide, Adelaide SA 5005, Australia}
\email{finnur.larusson@adelaide.edu.au}

\subjclass[2010]{Primary 32E10.  Secondary 32H02, 32Q28, 54C35, 54C55, 55M15.}

\keywords{Stein manifold, Oka manifold, parametric Oka property, deformation retract, absolute neighbourhood retract, mixed model structure.}

\date{17 June 2013.  Latest minor changes 19 August 2013.}

\thanks{The author was supported by Australian Research Council grant DP120104110.}

\thanks{The author is grateful to Jaka Smrekar for helpful discussions.}

\begin{document}

\begin{abstract}  
The basic result of Oka theory, due to Gromov, states that every continuous map $f$ from a Stein manifold $S$ to an elliptic manifold $X$ can be deformed to a holomorphic map.  It is natural to ask whether this can be done for all $f$ at once, in a way that depends continuously on $f$ and leaves $f$ fixed if it is holomorphic to begin with.  In other words, is $\scrO(S,X)$ a deformation retract of $\scrC(S,X)$?  We prove that it is if $S$ has a strictly plurisubharmonic Morse exhaustion with finitely many critical points; in particular, if $S$ is affine algebraic.  The only property of $X$ used in the proof is the parametric Oka property with approximation with respect to finite polyhedra, so our theorem holds under the weaker assumption that $X$ is an Oka manifold.  Our main tool, apart from Oka theory itself, is the theory of absolute neighbourhood retracts.  We also make use of the mixed model structure on the category of topological spaces.
\end{abstract}

\maketitle

\noindent
The basic result of Oka theory, due to Gromov \cite{Gromov}, states that every continuous map $f$ from a Stein manifold $S$ to an elliptic manifold $X$ can be deformed to a holomorphic map.  It is natural to ask whether this can be done for all $f$ at once, in a way that depends continuously on $f$ and leaves $f$ fixed if it is holomorphic to begin with.  In other words, is $\scrO(S,X)$ a deformation retract\footnote{Following Hatcher \cite{Hatcher} and May \cite{May}, we use the term \emph{deformation retract} for what more commonly used to be called a \emph{strong deformation retract}.} of $\scrC(S,X)$?  Our main result, valid under the weaker assumption that $X$ is an Oka manifold, is the following.

\begin{theorem}  \label{t:main-result}
Let $X$ be an Oka manifold and let $S$ be a Stein manifold with a strictly plurisubharmonic Morse exhaustion with finitely many critical points.  Then $\scrO(S,X)$ is a deformation retract of $\scrC(S,X)$.
\end{theorem}

Theorem \ref{t:main-result} follows from Propositions \ref{p:key} and \ref{p:C-is-ANR} and Theorem \ref{t:O-is-ANR} below.

\begin{remark}  \label{r:hypothesis}
(a)  The assumption on $S$ is satisfied if $S$ is affine algebraic.  Namely, algebraically embed $S$ into $\C^n$ for some $n$.  For a generic choice of $a\in\C^n$, the smooth strictly plurisubharmonic exhaustion $S\to[0,\infty)$, $z\mapsto\lVert z-a\rVert^2$, is a Morse function with finitely many critical points.

(b)  The properties of $S$ that we use to prove Theorem \ref{t:main-result} are, first, that $S$ is finitely dominated, meaning that $\id_S$ is homotopic to a map with relatively compact image (see Proposition \ref{p:C-is-ANR}) and, second, that for every compact subset $K$ of $S$, there are compact subsets $K\subset\cdots\subset K_1\subset L_1\subset K_0\subset L_0$ such that for every $n\geq 0$, $K_n\subset L_n^\circ$, $K_n$ is holomorphically convex, and $S$ deformation-retracts onto $L_n$ (see Theorem \ref{t:O-is-ANR}).  Both properties follow from the hypothesis that $S$ has a strictly plurisubharmonic Morse exhaustion with finitely many critical points.

(c)  Theorem \ref{t:main-result} provides what we might call \emph{local} versions of well-known Oka properties.  The simplest case is as follows.  If $S$ is Stein and $X$ is Oka, then every path in $\scrC(S,X)$ joining two points in $\scrO(S,X)$ can be deformed, keeping the end points fixed, to a path in $\scrO(S,X)$.  When $S$ has a strictly plurisubharmonic Morse exhaustion with finitely many critical points, Theorem \ref{t:main-result} implies that every neighbourhood $U$ in $\scrC(S,X)$ of a point in $\scrO(S,X)$ contains a neighbourhood $V$ such that every path in $V$ joining two points in $\scrO(S,X)$ can be deformed within $U$, keeping the end points fixed, to a path in $\scrO(S,X)$.
\end{remark}

We say that a pair $(S,X)$ of complex manifolds satisfies the \emph{parametric Oka property} with respect to a topological space $P$ and a subspace $Q$ of $P$, or with respect to the inclusion $Q\hookrightarrow P$, if for every commuting square
\[ \xymatrix{ Q \ar[r] \ar@{^{(}->}[d] & \scrO(S,X) \ar@{^{(}->}[d] \\ P \ar[r]^<<<<{f} & \scrC(S,X) } \]
of continuous maps, $f$ can be deformed, keeping $f\vert Q$ fixed, until the square has a lifting, that is, until the deformed $f$ maps $P$ into $\scrO(S,X)$.  More explicitly, there is a continuous map $F:P\times[0,1]\to\scrC(S,X)$ such that:
\begin{itemize}
\item  $F(\cdot,0)=f$,
\item  $F(\cdot,t)\vert Q=f\vert Q$ for all $t\in[0,1]$,
\item  $F(P\times\{1\})\subset\scrO(S,X)$.
\end{itemize}

Here, $\scrC(S,X)$ is equipped with the compact-open topology, in which $\scrO(S,X)$ is a closed subspace.  In the following, we shall often abbreviate $\scrC(S,X)$ as $\scrC$, and $\scrO(S,X)$ as $\scrO$.  We will refer to $P$ as a \emph{parameter space}.  Note that we are considering the parametric Oka property without approximation and interpolation.  I hope to include these in future work.

One of the several equivalent defining properties of an Oka manifold $X$ is the parametric Oka property with approximation (or interpolation or both) with respect to finite polyhedra and subpolyhedra for maps from an arbitrary Stein manifold $S$ to $X$.  In fact, Oka manifolds are known to satisfy this property with respect to compact subsets of Euclidean spaces \cite[Thm.~5.4.4]{Forstneric2011}, and elliptic manifolds are known to satisfy it with respect to all compact topological spaces \cite[Thm.~6.2.2]{Forstneric2011}\footnote{The theorem is stated for compact Hausdorff parameter spaces, but since $\scrO$ and $\scrC$ are Hausdorff, it holds for all compact parameter spaces.}.  Elliptic manifolds are Oka, and Stein Oka manifolds are elliptic.  No Oka manifolds are known not to be elliptic, but many Oka manifolds are not known to be elliptic.  Complex Lie groups and their homogeneous spaces are elliptic; for more examples of Oka manifolds, see \cite[Sec.~5.5]{Forstneric2011}.  For more background on Oka theory, see \cite{Forstneric2011} and \cite{Forstneric-Larusson}.  Our question is whether every pair $(S,X)$, where $S$ is Stein and $X$ is Oka, satisfies the parametric Oka property with respect to the universal parameter spaces $\scrO(S,X)\hookrightarrow\scrC(S,X)$.

We remind the reader of the two classical model structures on the category of topological spaces that have been known since about 1970.  In the h-structure of Str\o m, the weak equivalences are the homotopy equivalences, and the fibrations and cofibrations are those of Hurewicz.  In the q-structure of Quillen, the weak equivalences are the weak homotopy equivalences, the fibrations are those of Serre, and the cofibrations are retracts of relative generalised cell complexes.    We will make use of a third structure, the mixed structure or m-structure of Cole \cite{Cole}, which appeared only in 2006, in which the weak equivalences are the weak homotopy equivalences, and the fibrations are those of Hurewicz.  The m-cofibrant spaces are those with the homotopy type of a CW complex.  For more background, see \cite[Ch.~17]{May-Ponto}.

A proof of the following result is contained in \cite[Sec.~16]{Larusson2004}.

\begin{theorem}  \label{t:my-old-result}
For manifolds $S$ and $X$, the following are equivalent.
\begin{enumerate}
\item[(i)]  $(S,X)$ satisfies the parametric Oka property with respect to finite polyhedra and subpolyhedra. 
\item[(ii)]  $(S,X)$ satisfies the parametric Oka property with respect to q-cofibrations $Q\hookrightarrow P$, where $P$ and $Q$ are q-cofibrant.
\item[(iii)]  The inclusion $\scrO(S,X)\hookrightarrow\scrC(S,X)$ is a q-equivalence.
\end{enumerate}
\end{theorem}

The nontrivial implication is (iii) $\Rightarrow$ (ii).  It is used to prove Proposition \ref{p:weak-result} below.  A special case of (ii) is when $P$ is a CW complex and $Q$ is a subcomplex of $P$.  This does not answer our question, as $\scrC$ carries a CW structure only in trivial cases: it is metrisable, but a metrisable CW complex is locally compact.

The following proposition provides a convenient reformulation of $\scrO$ being a deformation retract of $\scrC$.

\begin{proposition}  \label{p:reformulation}
For manifolds $S$ and $X$, the following are equivalent.
\begin{enumerate}
\item[(i)]  $(S,X)$ satisfies the parametric Oka property with respect to every topological space $P$ and every subspace of $P$.
\item[(ii)]  $\scrO(S,X)$ is a deformation retract of $\scrC(S,X)$.
\item[(iii)]  The inclusion $\scrO(S,X)\hookrightarrow\scrC(S,X)$ is an acyclic cofibration in the h-structure or, equivalently, in the m-structure.
\end{enumerate}
\end{proposition}

Note that (iii) is weaker than requiring the inclusion to be an acyclic cofibration in the q-structure.

\begin{proof}
(i) $\Leftrightarrow$ (ii) is obvious.

(ii) $\Rightarrow$ (iii):  Evidently, if $\scrO$ is a deformation retract of $\scrC$, then the inclusion $\scrO\hookrightarrow\scrC$ is a homotopy equivalence.  Since $\scrC$ is metrisable, the inclusion is also an h-cofibration \cite[Thm.~4.1.14]{Aguilar-et-al}.

(iii) $\Rightarrow$ (ii):  See \cite[Sec.~6.5]{May}.
\end{proof}

Our problem is thus reduced to finding a reasonable sufficient condition for the inclusion $\scrO\hookrightarrow\scrC$ to be an acyclic cofibration in the m-structure when $S$ is Stein and $X$ is Oka.  We know that $\scrO\hookrightarrow\scrC$ is a q-equivalence, so we need $\scrO\hookrightarrow\scrC$ to be an m-cofibration.  The following is our key observation.

\begin{proposition}  \label{p:key}
Let $S$ be a Stein manifold and $X$ be an Oka manifold.  If $\scrO(S,X)$ and $\scrC(S,X)$ are ANR, then $\scrO(S,X)$ is a deformation retract of $\scrC(S,X)$.
\end{proposition}

For the theory of ANR (absolute neighbourhood retracts for metric spaces), we refer the reader to \cite{Hu} and \cite{vanMill}.  See also \cite{Fritsch-Piccinini} and \cite{Weber}.  Let us mention that an ANR is locally contractible, meaning that every neighbourhood $U$ of each point contains a neighbourhood which is contractible in $U$.  In particular, an ANR is locally path connected.\footnote{Let us review a few facts about the relationship between the class of ANRs and the class of CW complexes.  They intersect in the class of locally finite ($=$ locally compact $=$ metrisable $=$ first countable) CW complexes.  A topological space has the homotopy type of an ANR if and only if it has the homotopy type of a CW complex (that is, is m-cofibrant).  A metrisable space is ANR if and only if every open subset has the homotopy type of a CW complex \cite{Cauty}.}

\begin{proof}
Since $\scrO$ and $\scrC$ are ANR, the inclusion $\scrO\hookrightarrow\scrC$ is an h-cofibration \cite[Thm.~4.2.15]{Aguilar-et-al}.  Since an ANR is m-cofibrant, $\scrO\hookrightarrow\scrC$ is in fact an m-cofibration \cite[Cor.~3.12]{Cole}.
\end{proof}

A good sufficient condition for $\scrC$ to be ANR is available.  Before stating it, let us note that if $\scrC$ is ANR, then for $\scrO$ to be a deformation retract of $\scrC$ is a \emph{local topological property} of $\scrO$.

\begin{corollary}  \label{c:equivalence}
Let $S$ be a Stein manifold and $X$ be an Oka manifold, such that $\scrC(S,X)$ is ANR.  Then $\scrO(S,X)$ is a deformation retract of $\scrC(S,X)$ if and only if $\scrO(S,X)$ is ANR.
\end{corollary}

\begin{proof}
If $\scrC$ is ANR, then $\scrO\hookrightarrow\scrC$ is an h-cofibration if and only if $\scrO$ is ANR \cite[Thm.~4.2.15]{Aguilar-et-al}.
\end{proof}

It is well known that the following are equivalent for a CW complex $S$.
\begin{enumerate}
\item[(i)]  $S$ has the homotopy type of a compact topological space.
\item[(ii)]  $S$ has the homotopy type of a compact metric space.
\item[(iii)]  $S$ is \textit{finitely dominated}, meaning that there is a finite CW complex $F$ and continuous maps $\phi:S\to F$ and $\psi:F\to S$ such that $\psi\circ\phi$ is homotopic to $\id_S$.  It follows that the homotopy groups of $S$ are finitely presented.
\item[(iv)]  The product of $S$ and the circle has the homotopy type of a finite CW complex.
\item[(v)]  $\id_S$ is homotopic to a map with relatively compact image.
\end{enumerate}
The implication (ii) $\Rightarrow$ (i) is trivial, (i) $\Rightarrow$ (v) $\Rightarrow$ (iii) are easy, but (iii) $\Rightarrow$ (ii) \cite{Ferry} and (iii) $\Leftrightarrow$ (iv) \cite{Mather} are nontrivial.

It is well understood when a finitely dominated CW complex $S$ has the homotopy type of a finite CW complex:  it does if and only if a K-theoretic obstruction due to Wall vanishes.  The obstruction vanishes when $S$ is simply connected, but not in general.  A finitely dominated CW complex has the homotopy type of a countable finite-dimensional CW complex, so it has the homotopy type of a Stein manifold.  Hence there are many finitely dominated Stein manifolds that do not have the homotopy type of a finite CW complex.  I do not know any explicit examples of such manifolds.

\begin{proposition}  \label{p:C-is-ANR}
Let $S$ be a finitely dominated countable CW complex and let $X$ be a locally finite countable CW complex.  Then $\scrC(S,X)$ is ANR.
\end{proposition}

In particular, $\scrC(S,X)$ is ANR if $S$ and $X$ are manifolds and $S$ is finitely dominated.  

\begin{proof}
By assumption, $S$ is homotopy equivalent to a compact metric space $M$.  Now $\scrC(M,X)$ has the homotopy type of a countable CW complex \cite[Cor.~2]{Milnor}.  Since $\scrC(S,X)$ and $\scrC(M,X)$ are homotopy equivalent, $\scrC(S,X)$ has the homotopy type of a CW complex.  Being a locally finite CW complex, $X$ is ANR.  By \cite[Thm.~1.1]{Smrekar-Yamashita}, it follows that $\scrC(S,X)$ is ANR.
\end{proof}

Under very mild additional assumptions, it even follows from \cite[Thm.~1.2]{Smrekar-Yamashita} that $\scrC(S,X)$ is an $\ell_2$-manifold and hence homeomorphic to an open subset of $\ell_2$.

We can now prove that if a Stein manifold $S$ is merely finitely dominated, then every continuous map $f$ from $S$ to an Oka manifold can be deformed to a holomorphic map in a way that depends continuously on $f$, but with no guarantee that $f$ will be left fixed if it is holomorphic.  In other words, the inclusion $\scrO\hookrightarrow\scrC$ has a left homotopy inverse.

\begin{proposition}  \label{p:weak-result}
Let $S$ be a finitely dominated Stein manifold and let $X$ be an Oka manifold.  Then $\id_{\scrC(S,X)}$ can be deformed to a map with image in $\scrO(S,X)$.
\end{proposition}

\begin{proof}
By the proof of Proposition \ref{p:C-is-ANR}, $\scrC$ has the homotopy type of a CW complex $P$.  Let $\alpha:P\to\scrC$ be a homotopy equivalence with a homotopy inverse $\beta$.  By Theorem \ref{t:my-old-result}, $(S,X)$ has the parametric Oka property with respect to $P$, so $\alpha$ can be deformed to a map $\alpha'$ taking $P$ into $\scrO$.  Then $\id_{\scrC}$ is homotopic to $\alpha'\circ\beta$.
\end{proof}

Next we derive a sufficient condition for spaces of holomorphic maps to be ANR.

\begin{theorem}  \label{t:O-is-ANR}
Let $X$ be an Oka manifold and let $S$ be a Stein manifold with a strictly plurisubharmonic Morse exhaustion with finitely many critical points.  Then $\scrO(S,X)$ is ANR.
\end{theorem}

\begin{proof}
We shall verify the Dugundji-Lefschetz characterisation of the ANR property for $\scrO$ \cite[Thm.~IV.4.1]{Hu}, \cite[Thm.~5.2.1]{vanMill}.  (As far as I can see, this is the only feasible approach to the problem.)  Let $\mathscr U$ be an open cover of $\scrO$.  We need to produce a refinement $\mathscr V$ of $\mathscr U$ such that if $P$ is a simplicial complex\footnote{\cite[Thm.~5.2.1]{vanMill} is stated for simplicial complexes that are countable and locally finite, but these properties are irrelevant here.} with a subcomplex $Q$ containing all the vertices of $P$, then every continuous map $\phi_0:Q\to\scrO$ such that for each simplex $\sigma$ of $P$, $\phi_0(\sigma\cap Q)\subset V$ for some $V\in\mathscr V$, extends to a continuous map $\phi:P\to\scrO$ such that for each simplex $\sigma$ of $P$, $\phi(\sigma)\subset U$ for some $U\in\mathscr U$.

Let $f\in\scrO$.  Choose $U\in\mathscr U$ with $f\in U$.  Standard methods show that there is a holomorphic vector bundle $E$ over $S$ such that a neighbourhood $W$ of the graph of $f$ in $S\times X$ is biholomorphic over $S$ to a neighbourhood  of the zero section in $E$, with the graph of $f$ corresponding to the zero section.  (This requires the Stein property of $S$ but not the Oka property of $X$.)  By shrinking the latter neighbourhood, we may assume that its intersection with each fibre of $E$ is convex.  By shrinking the former neighbourhood, we may assume that there is a compact set $K\subset S$ such that if the graph of $g\in\scrO$ over $K$ lies in $W$, then $g\in U$.

Find compact subsets $K\subset\cdots\subset K_1\subset L_1\subset K_0\subset L_0$ of $S$ such that for every $n\geq 0$, $K_n\subset L_n^\circ$, $K_n$ is holomorphically convex, and $S$ deformation-retracts onto $L_n$ (see Remark \ref{r:hypothesis}).  Let $V_n$ be the neighbourhood of $f$ consisting of all maps in $\scrO$ whose graph over $L_n$ lies in $W$.  Let the refinement $\mathscr V$ of $\mathscr U$ consist of all such open sets $V_0$, one for each $f\in\scrO$.

Let $P$, $Q$, and $\phi_0$ be as above, so for each simplex $\sigma$ of $P$, $\phi_0(\sigma\cap Q)\subset V_0^\sigma$ for some $V_0^\sigma\in\mathscr V$ associated to some $f\in\scrO$ as above.  Adorn the other associated sets $W$, $K_n$, $L_n$, and $V_n$, and the associated bundle $E$ with a superscript $\sigma$ as well.  Let $P_n=P^n\cup Q$, where $P^n$ is the $n$-skeleton of $P$.  We shall construct continuous maps $\phi_n:P_n\to\scrO$, $n\geq 1$, such that $\phi_n\vert P_{n-1}=\phi_{n-1}$ and for every simplex $\sigma$ of $P$, $\phi_n(\sigma\cap P_n)\subset V_n^\sigma$.  The map $\phi:P\to\scrO$ with $\phi\vert P_n=\phi_n$ for each $n\geq 0$ will then be continuous with respect to the Whitehead topology on $P$, and for each simplex $\sigma$ of $P$, $\phi(\sigma)\subset U$ for some $U\in\mathscr U$.

Suppose $\phi_n$, $n\geq 0$, is given.  Let $\sigma$ be an $(n+1)$-simplex of $P$ but not of $Q$.  The interior $\sigma\setminus\partial\sigma$ of $\sigma$ does not intersect $P_n$.  Note also that the interiors of distinct $(n+1)$-simplices do not intersect.  We need to suitably extend $\phi_n\vert\partial\sigma$ to $\sigma$.  We have $\phi_n(\partial\sigma)\subset \phi_n(\sigma\cap P_n)\subset V_n^\sigma$.  Let $C$ be the set of maps with graph in $W^\sigma$ in the space of continuous maps $L_n^\sigma\to X$ that are holomorphic on $(L_n^\sigma)^\circ$.  Then $C$ is homeomorphic to a convex subset of the Banach space of continuous sections of $E^\sigma$ over $L_n^\sigma$, so $C$ is contractible.  Hence the composition of $\phi_n:\partial\sigma\to V_n^\sigma$ and the restriction map $V_n^\sigma \to C$ extends by the cofibration $\partial\sigma\hookrightarrow\sigma$ to a continuous map $\alpha:\sigma\to C$.

We may view $\alpha$ as a continuous map $S\times\partial\sigma \cup L_n^\sigma\times\sigma \to X$.  Now $L_n^\sigma$ is a deformation retract of $S$, so $S\times\partial\sigma \cup L_n^\sigma\times\sigma$ is a deformation retract of $S\times\sigma$ \cite[Thm.~6]{Strom}.  Thus $\alpha$ extends to a continuous map $\alpha:S\times\sigma\to X$.  Note that $\alpha(\cdot,t)$ is holomorphic on $S$ for all $t\in\partial\sigma$, and $\alpha(\cdot,t)$ is holomorphic on the neighbourhood $(L_n^\sigma)^\circ$ of $K_n^\sigma$ for all $t\in\sigma$.  

The parametric Oka property with approximation with respect to finite polyhedra, one of the equivalent formulations of the Oka property of $X$ \cite[Thm.~5.4.4, Sec.~5.15]{Forstneric2011}, now implies that $\alpha$ may be deformed to a continuous map $\beta:S\times\sigma\to X$ such that $\beta(\cdot,t)=\alpha(\cdot,t)$ for all $t\in\partial\sigma$, $\beta(\cdot,t)$ is holomorphic on $S$ for all $t\in\sigma$, and $\beta$ uniformly approximates $\alpha$ as closely as desired on $K_n^\sigma\times\sigma$.  If the approximation is close enough, then defining $\phi_{n+1}(t)=\beta(\cdot,t)$ for $t\in\sigma$ gives a continuous extension of $\phi_n\vert\partial\sigma$ to $\sigma$ with $\phi_{n+1}(\sigma)\subset V_{n+1}^\sigma$.
\end{proof}

Theorem \ref{t:main-result} now follows from Propositions \ref{p:key} and \ref{p:C-is-ANR} and Theorem \ref{t:O-is-ANR}.  Without any difficult arguments in Oka theory, using mainly topology, we have shown that for many Stein sources, including all affine algebraic ones, the parametric Oka property with approximation with respect to finite polyhedra implies the parametric Oka property (without approximation) with respect to completely arbitrary parameter spaces.

\begin{remark}  \label{r:ell-2-manifold}
The prototypical Oka manifold is a complex Lie group $G$.  If $\scrO(S,G)$ is ANR, then it is an $\ell_2$-manifold and hence homeomorphic to an open subset of $\ell_2$ (except in the trivial case when it is a Lie group) \cite[Cor.~1]{Dobrowolski-Torunczyk}.  I do not know whether this holds for Oka manifolds in general.
\end{remark}

If $S$ is not finitely dominated, ANR theory does not apply.  For example, $\scrO(\C\setminus\mathbb N, \C^*)$ and $\scrC(\C\setminus\mathbb N, \C^*)$ are not semilocally contractible, so they are not m-cofibrant, let alone ANR.  In this very particular case, though, we can show that $\scrO$ is a deformation retract of $\scrC$.  We conclude the paper by explaining this.

Let $S$ be a Stein manifold.  Being Fr\'echet spaces, $\scrC(S,\C^n)$ and $\scrO(S,\C^n)$ are ANR, so $\scrO(S,\C^n)$ is a deformation retract of $\scrC(S,\C^n)$.  Let $\scrC_0(S,\C^*)$ denote the component---connected component or path component: they are the same---of $\scrC(S,\C^*)$ consisting of the nullhomotopic maps.  If $S$ is simply connected, then $\scrC_0(S,\C^*)=\scrC(S,\C^*)$.  Define $\scrO_0(S,\C^*)$ similarly.  The universal covering space of $\scrC_0(S,\C^*)$ is $\scrC(S,\C)$, so the two are locally homeomorphic.  It follows that $\scrC_0(S,\C^*)$ is ANR.  So is $\scrO_0(S,\C^*)$.  Hence, $\scrO_0(S,\C^*)$ is a deformation retract of $\scrC_0(S,\C^*)$.  The same holds for every other component of $\scrC(S,\C^*)$ and the unique component of $\scrO(S,\C^*)$ that it contains.  More generally, this argument proves the following result.

\begin{proposition}  \label{p:quotients-of-affine-spaces}
Let $X$ be the quotient of $\C^n$ by a discrete subgroup.  If $S$ is a Stein manifold, then $(S,X)$ satisfies the parametric Oka property with respect to every connected topological space $P$ and every subspace of $P$.  If $S$ is simply connected, then $(S,X)$ satisfies the parametric Oka property with respect to every topological space $P$ and every subspace of $P$.
\end{proposition}

We now restrict ourselves to the case when $S=\C\setminus\mathbb N$ and $X=\C^*$.  The bijection $\scrO/\scrO_0\to\mathbb Z^\mathbb N$ induced by the divisor map $\scrO\to\mathbb Z^\mathbb N$ taking a map to its sequence of winding numbers about each puncture of $\C\setminus\mathbb N$ is a homeomorphism with respect to the quotient topology on the source and the product topology on the target.  So is the analogous map $\scrC/\scrC_0\to\mathbb Z^\mathbb N$.  As observed above, $\scrO_0$ is a deformation retract of $\scrC_0$.  It follows that $\scrO$ is a deformation retract of $\scrC$ if we can show that $\scrO$ is homeomorphic to $\scrO_0\times\mathbb Z^\mathbb N$ over $\mathbb Z^\mathbb N$, and similarly for $\scrC$.  This, in turn, follows if we can produce a continuous section of the projection $\scrO\to\mathbb Z^\mathbb N$.

The values of our section will be Weierstrass products.  The theory of Weierstrass products is of course classical and well known, but since I have never seen the continuity property needed here explicitly stated, I shall give some details.  Recall the Weierstrass factors
\[ E_0(z)=1-z, \qquad E_m(z)=(1-z)\exp\sum_{j=1}^m \frac{z^j}j, \quad m\in\mathbb N, \]
and the basic estimate $\lvert E_m(z)-1\rvert\leq\lvert z\rvert^{m+1}$ if $\lvert z\rvert\leq 1$.  For each $\nu\in\mathbb Z^\mathbb N$ and $k\in\mathbb N$, let $m_k^\nu$ be the sum of $k$ and the smallest nonnegative integer no smaller than $\log\lvert\nu(k)\rvert$.  For $k\geq e\lvert z\rvert$,
\[ \left(\frac{\lvert z\rvert} k\right)^{m_k^\nu}\frac{\lvert\nu(k)\rvert}k \leq e^{-m_k^\nu}\frac{\lvert\nu(k)\vert}k \leq \frac 1{ke^k}, \]
so the series $\sum\limits_{k=1}^\infty \left(\dfrac z k\right)^{m_k^\nu}\dfrac{\nu(k)}k$ is absolutely locally uniformly convergent in $\C$.  Since
\[ \frac{dE_m^n(z/k)/dz}{E_m^n(z/k)}=\left(\frac z k\right)^m\frac n{z-k}, \]
it follows that the product
\[ f_\nu(z)=\prod_{k=1}^\infty E_{m_k^\nu}^{\nu(k)}(z/k) \]
converges absolutely locally uniformly in $\C$ to a meromorphic function with divisor $\nu$.  Hence $\nu\mapsto f_\nu$ is a section of the projection $\scrO\to\mathbb Z^\mathbb N$.  

We claim that the section is continuous.  Let $\nu\in\mathbb Z^\mathbb N$, let $K$ be a compact subset of $\C\setminus\mathbb N$, and let $\epsilon>0$.  We need to show that there is $k_0\in\mathbb N$ such that if $\mu\in\mathbb Z^\mathbb N$ and $\mu(k)=\nu(k)$ for all $k<k_0$, then $\lVert f_\mu - f_\nu\rVert_K < \epsilon$.  Here, $\lVert\cdot\rVert$ denotes the supremum norm.

Find $c>0$ such that $\left\lVert \prod\limits_{k=1}^n E_{m_k^\nu}^{\nu(k)}(\cdot/k)\right\rVert_K<c$ for all $n\in\mathbb N$.  Choose $k_0$ such that $k_0\geq e\lvert z\rvert$ for all $z\in K$.  Let $\lambda\in\mathbb Z^\mathbb N$, $k\geq k_0$, and $z\in K$.  Then $\lvert E_{m_k^\lambda}(z/k)-1\rvert\leq e^{-2}$ by the basic estimate.  Let $\log E_{m_k^\lambda}(z/k)$ be the principal logarithm.  Now $\lvert\log(1+w)\rvert\leq\tfrac 3 2\lvert w\rvert$ if $\lvert w\rvert\leq\tfrac 1 2$, so by the basic estimate,
\[ \lvert\log E_{m_k^\lambda}(z/k)\rvert\leq\tfrac 3 2\lvert E_{m_k^\lambda}(z/k)-1\rvert\leq \tfrac 3 2\lvert z/k\rvert^{m_k^\lambda +1}\leq \tfrac 3 2 e^{-(m_k^\lambda +1)}, \]
and
\[ \sum_{k=k_0}^\infty\lvert\lambda(k)\log E_{m_k^\lambda}(z/k)\rvert \leq \tfrac 3 2 \sum_{k=k_0}^\infty \lvert\lambda(k)\rvert e^{-(m_k^\lambda +1)} \leq \tfrac 3 2\sum_{k=k_0}^\infty e^{-(k+1)}.\]
The sum on the right can be made arbitrarily small by taking $k_0$ sufficiently large.  Hence, after replacing $k_0$ by a larger number if necessary, we have $\lVert f_\lambda-1\rVert_K<\epsilon/(2c)$ for all $\lambda\in\mathbb Z^\mathbb N$ with $\lambda(k)=0$ for $k<k_0$.

Finally, say $\mu\in\mathbb Z^\mathbb N$ with $\mu(k)=\nu(k)$ for all $k<k_0$.  Let $\tilde\mu(k)=0$ for $k<k_0$ and $\tilde\mu(k)=\mu(k)$ for $k\geq k_0$.  Define $\tilde\nu$ similarly.  Then
\[ \lVert f_\mu-f_\nu\rVert_K \leq \left\lVert \prod\limits_{k=1}^{k_0-1} E_{m_k^\nu}^{\nu(k)}(\cdot/k)\right\rVert_K \lVert f_{\tilde\mu}-f_{\tilde\nu}\rVert_K <\epsilon. \]


\begin{thebibliography}{99}

\bibitem{Aguilar-et-al}
Aguilar, M., S.\ Gitler, and C.\ Prieto.  \textit{Algebraic topology from a homotopical viewpoint.}  Universitext.  Springer-Verlag, 2002.

\bibitem{Cauty}
Cauty, R.  \textit{Une caract\'erisation des r\'etractes absolus de voisinage.}  Fund.\ Math.\ \textbf{144} (1994) 11--22.

\bibitem{Cole}
Cole, M.  \textit{Mixing model structures.}  Topology Appl.\ \textbf{153} (2006) 1016--1032. 

\bibitem{Dobrowolski-Torunczyk}
Dobrowolski, T. and H.\ Toru\'nczyk.  \textit{Separable complete ANRs admitting a group structure are Hilbert manifolds.}  Topology Appl.\ \textbf{12} (1981) 229--235.

\bibitem{Ferry}
Ferry, S.  \textit{Homotopy, simple homotopy and compacta.}  Topology \textbf{19} (1980) 101--110. 

\bibitem{Forstneric2011}
Forstneri\v c, F.  \textit{Stein manifolds and holomorphic mappings.}  Ergebnisse der Mathematik und ihrer Grenzgebiete.  3.\ Folge, 56.  Springer-Verlag, 2011.

\bibitem{Forstneric-Larusson}
Forstneri\v c, F.\ and F.\ L\'arusson.  \textit{Survey of Oka theory.}  New York J.\ Math.\ \textbf{17a} (2011) 11--38.

\bibitem{Fritsch-Piccinini}
Fritsch, R.\ and R.\ A.\ Piccinini.  \textit{Cellular structures in topology.}  Cambridge Studies in Advanced Mathematics, 19.  Cambridge University Press, 1990.

\bibitem{Gromov}
Gromov, M.  \textit{Oka's principle for holomorphic sections of elliptic bundles.}  J.\ Amer.\ Math.\ Soc.\ \textbf{2} (1989) 851--897.

\bibitem{Hatcher}
Hatcher, A.  \textit{Algebraic topology.}  Cambridge University Press, 2002. 

\bibitem{Hu}
Hu, S.-T.  \textit{Theory of retracts.}  Wayne State University Press, 1965.

\bibitem{Larusson2004}
L\'arusson, F.  \textit{Model structures and the Oka principle.}  J.\ Pure Appl.\ Algebra \textbf{192} (2004) 203--223.

\bibitem{Mather}
Mather, M.  \textit{Counting homotopy types of manifolds.}  Topology \textbf{3} (1965) 93--94. 

\bibitem{May}
May, J.\ P.  \textit{A concise course in algebraic topology.}  Chicago Lectures in Mathematics.  University of Chicago Press, 1999.

\bibitem{May-Ponto}
May, J.\ P.\ and K.\ Ponto.  \textit{More concise algebraic topology. Localization, completion, and model categories.}  Chicago Lectures in Mathematics.  University of Chicago Press, 2012.
 
\bibitem{Milnor}
Milnor, J.  \textit{On spaces having the homotopy type of a CW-complex.}  Trans.\ Amer.\ Math.\ Soc.\ \textbf{90} (1959) 272--280. 
 
\bibitem{Smrekar-Yamashita}
Smrekar, J.\ and A.\ Yamashita.  \textit{Function spaces of CW homotopy type are Hilbert manifolds.}  Proc.\ Amer.\ Math.\ Soc.\ \textbf{137} (2009) 751--759.

\bibitem{Strom}
Str\o m, A.  \textit{Note on cofibrations II.}  Math.\ Scand.\ \textbf{22} (1968) 130--142.

\bibitem{vanMill}
van Mill, J.  \textit{Infinite-dimensional topology.  Prerequisites and introduction.}  North-Holland Mathematical Library, 43.  North-Holland Publishing Co., 1989.

\bibitem{Weber}
Weber, C.  \textit{Quelques th\'eor\`emes bien connus sur les A.N.R.\ et les C.W.\ complexes.}  Enseign.\ Math.\ (2) \textbf{13} (1967) 211--222.

\end{thebibliography}
\end{document}